\documentclass[preprint]{imsart}

\usepackage{amsfonts}
\usepackage{amsmath,bm}
\usepackage{mathtools,mathrsfs,nicefrac,graphicx}
\usepackage{amsthm,nicefrac,bigints}
\usepackage[utf8]{inputenc}
\usepackage[T1]{fontenc}
\usepackage{epsfig}
\usepackage{bigints}
\usepackage{multirow}
\usepackage{eurosym}
\usepackage{enumerate}
\usepackage{enumitem}
\usepackage{bbm,bm}
\usepackage{xcolor}

\usepackage{array}
\newcolumntype{L}[1]{>{\raggedright\let\newline\\\arraybackslash\hspace{0pt}}m{#1}}
\newcolumntype{C}[1]{>{\centering\let\newline\\\arraybackslash\hspace{0pt}}m{#1}}
\newcolumntype{R}[1]{>{\raggedleft\let\newline\\\arraybackslash\hspace{0pt}}m{#1}}

\usepackage{float}
\usepackage{url}
\usepackage{hyperref}
\usepackage{marginnote}

\newtheorem{theorem}{Theorem} 
\newtheorem{proposition}[theorem]{Proposition}

\newtheorem*{theorem*}{Theorem}

\theoremstyle{definition}\newtheorem{remark}[theorem]{Remark}

\newcommand{\bN}{\mathbb{N}}
\newcommand{\bP}{\mathbb{P}}

\newcommand{\bR}{\mathbb{R}}

\newcommand{\bG}{\mathbb{G}}
\newcommand{\bS}{\mathbb{S}}

\newcommand{\cH}{\mathcal{H}}

\newcommand{\cL}{\mathcal{L}}

\newcommand{\cP}{\mathcal{P}}

\newcommand{\cW}{\mathcal{W}}

\newcommand{\Aut}{\text{\rm Aut}}
\newcommand{\aut}{\text{\rm aut}}
\newcommand{\Ber}{\text{\rm Ber}}
\newcommand{\ER}{\text{\rm ER}}

\newcommand{\trace}{\text{\rm trace}}

\newcommand{\var}{\text{\rm var}}

\newcommand{\Bin}{\text{\rm Bin}}

\newcommand{\toas}{\to_{\text{\rm\tiny a.s.}}} 
\newcommand{\todistr}{\to_{\text{\rm\tiny d}}}

\newcommand{\toP}{\to_{\text{\rm\tiny P}}}

\newcommand{\ul}{^{\text{\rm\tiny ul}}}
\newcommand{\bc}{^{\text{\rm\tiny bc}}}

\newcommand{\brend}{\hfill $\triangleleft$} 

\allowdisplaybreaks

\begin{document}

\begin{frontmatter}

\title{Tests of graph homogeneity, subgraph counts, and quasirandomness}
\runtitle{Tests of graph homogeneity}

\begin{aug}

\author{\fnms{Rudolf} \snm{Gr{\"u}bel}\corref{}
        \ead[label=e2]{rgrubel@stochastik.uni-hannover.de}}
       \address[]{Institute of Actuarial and Financial Mathematics,
          Leibniz Universit\"at Hannover\\ Welfengarten 1,          
          D-30167 Hannover, Germany \printead{e2}}
\runauthor{R. Gr{\"u}bel}
\end{aug}

\begin{abstract}

Homogeneous random graphs, also known as Erd{\H os}-R\'enyi 
graphs, are a subset of the family of dense random graphs, specified by a graphon. 
We analyze several goodness-of-fit tests for these models  that are based on subgraph counts. 
To obtain the limiting null distribution of the test statistics we use a decomposition 
of graph functionals, which reveals a cancellation effect. Motivated by a quasirandomness result we
obtain a test that is consistent against all alternatives,, and we
use two popular parametric subfamilies to evaluate the other tests. 
The theoretical results refer to the limit $n\to \infty$ of the size $n$ of the graph,
the behavior for finite $n$ is illustrated by simulations.  
\end{abstract}



\end{frontmatter}

\date{\today}


\section{Introduction}\label{sec:intro}
Let $(G_n)_{n\in\bN}$ be a sequence of random graphs with a distribution that is  param\-etrized by a 
graphon $W$; see~\cite{Lovasz} and Section~\ref{subsec:graphon} below. 
If $W$ is constant then we obtain the Erd{\H o}s-R\'enyi graphs,  where each of the $n(n-1)/2$ potential edges 
of a simple graph with $n$ nodes is included in its edge set independently and with the same probability $p$.
This may be seen as a form of homogeneity. In this paper we consider tests of the compound hypothesis 
that $W\equiv p$ almost surely for some $p\in (0,1)$.  
Such procedures are also known as goodness-of-fit tests, and as omnibus tests if they are consistent
against all alternatives.

Our approach to this problem is inspired by omnibus tests of independence
\cite{Hoeffding,Yana, BergsmaDassios} that are based on permutation patterns in rank plots. 
Such tests can be
related to the concept of quasirandomness. For permutations the corresponding results in~\cite{Chan}
led to the tests analyzed in~\cite{BaGrEJS3}; see also Section~\ref{sec:conclusion}. Here we similarly 
use a famous quasirandomness result for graphs from~\cite{CGW} to obtain a consistent test 
for graph homogeneity.

Subgraph counts are graph functionals, meaning that they are
invariant under node relabeling. For permutation tests of independence the asymptotic null 
distributions of the test statistics motivated by~\cite{Chan}  have been obtained in~\cite{BaGrEJS3}  
via the classical Hoeffding
decomposition for symmetric functions.  Here we use the fact that graph functionals form an 
algebra, together with  a graph version~\cite[Chap. 6]{JLR}  of the Hoeffding decomposition, 
to expose a cancellation effect that is the basis for our distributional results. 
Homogeneous graphs are built on a family of $n(n-1)/2$ independent random variables so that 
the standard distributional rate for e.g.\ subgraph counts is $n^{-1}$. Loosely speaking, 
the tests considered here are first order degenerate, with associated rates $n^{-3/2}$.

Graphs and networks have been a major topic in theoretical and applied research for many
years, with the amount of relevant literature increasing rapidly. Several excellent introductions,
including discussions of the connections to various fields, can be found in recent textbooks and 
research monographs, ranging from number theory~\cite{Zhao} to theoretical physics 
to the social sciences~\cite{Lovasz};  see also~\cite{BFJ} for a discussion 
of the applied relevance of graph statistics and the use of subgraph counts.
This area connects discrete mathematics, theoretical computer science, and
probability theory and statistics. In applications the  Erd{\H o}s-R\'enyi model,
or indeed dense graphs altogether, are often regarded as unrealistic as the number of edges
grows as the square of the number of vertices. This motivates models where $p$ varies with $n$,
see~\cite{Oua} and~\cite{BFJ}.
However, the case with $p$ fixed continues to be at the center of theoretical
advances; see e.g.\ the recent developments in large deviation theory~\cite{Chatter}.
The theoretical results are often on the level of the law of large numbers. For statistical tests, 
`second order results', essentially the transition to a central limit theorem,
are needed.  Here~\cite{JLR} is a standard reference, which mainly treats the Erd{\H o}s-R\'enyi
model with the extension that the parameter $p$ may depend on the size $n$ of the graph;
see also the discussion in~\cite{KaurRoellin}. 

Below we first fix some basic notation and collect some important tools. In Section~\ref{sec:results} we 
introduce several tests that are based on (nonlinear) combinations of counts of edges, of paths of length two, 
of  triangles, and of circles of circumference four, and we obtain the respective limiting null distributions. Theorem~\ref{thm:mainC4} 
presents an omnibus test. In Section~\ref{sec:altpar} we consider two popular parametric families, stochastic 
block models and models with a linear band graphon, with a view towards consistency of the other 
tests against a restricted set of alternatives.  In Section~\ref{sec:sim} we show the results of some 
simulations with finite sample size, and with emphasis on alternatives that are related to consistency 
properties of the tests. 
In the final section we give a very brief general outline, with hints to potential further research topics.

\section{Background and tools}\label{sec:background}
We first collect some notation from graph theory, see e.g.\ \cite[Section 1.2]{JLR}.
Next we recall some aspects of the limit theory of dense graphs, see the standard monograph~\cite{Lovasz}.
We then discuss the quasirandomness  result that is the basis for one of our tests, 
see e.g.~\cite[Chap.\ 3]{Zhao}. In the last subsection
we consider distributional limits in the Erd{\H o}s-R\'enyi case, following the approach in \cite[Chapter 6]{JLR} 
and \cite[Chapter 11]{JansonGHS}, which is based on an analogue for graph functionals
of the classical Hoeffding decomposition.

\subsection{Graphs}\label{subsec:graphs}
A (finite simple) graph is a pair $G=(V(G),E(G))$ of sets, with $V(G)$ the finite set of vertices of 
$G$ and $E(G)$ its edges, where $E(G)$ is a set of size two subsets of $V(G)$.  
We sometimes abbreviate $e\in E(G)$ to $e\in G$, and we write $v(G)=\# V(G)$ and $e(G)=\# E(G)$ 
for the number of vertices and edges of $G$.
Let $\bG(V)$ be the set of all such graphs with vertex set $V$, with $\bG(V)=\bG_n$ 
if $V=[n]:=\{1,\ldots, n\}$, $n\in\bN$, and $\bG :=\bigcup_{n\in\bN}\bG_n$. We use some standard 
abbreviations for specific graphs: $K_j$ is the complete graph on $j$ nodes, $P_j\in\bG_j$ is the path with edges
$\{1,2\},\ldots,\{j-1,j\}$,
$C_j\in\bG_j$ is the circle with circumference $j\ge 3$,
and $S_j\in\bG_j$ is the star with edges $\{ 1,i \}$, $i=2,3,\ldots,j$. 
In particular,  $K_2=P_2$ consists of a single edge, $K_3=C_3$ is known as the triangle, and 
$C_4:= ([4],\{\{1,2\},\{2,3\}, \{3,4\},\{1,4\}\})\in\bG_4$. Finally, we write $K_2+K_2$ for the graph 
on four nodes with two non-adjacent edges, and use the obvious extension to the disjoint union of 
other graphs, or to more than two summands.

The group $\bS(V)$ of all bijective functions (permutations) $\pi:V\to V$ acts on $\bG(V)$ via 
$\pi . G= (V(G), \bigl\{ \{\pi(v),\pi(w)\}:\, \{v,w\}\in E(G)\bigr\}\bigr)$. The automorphism  group 
$\Aut(G)$ with size $\aut(G)$ consists of all $\pi\in\bS(V)$ such that $\pi .G=G$;
for example, $\aut(K_2)=2$ and $\aut(C_4)=8$. 
The orbit of $G\in\bG(V)$ is the set $G\ul:=\{\pi.G:\, \pi\in\bS(V)\}$. These partition $\bG(V)$ into 
subsets of isomorphic graphs, and are commonly referred to as \emph{unlabeled graphs}.
By a (real-valued) \emph{graph functional} we mean a function $\Psi:\bG\to\bR$ that does not depend 
on the labeling, so that $\Psi(\pi.G)=\Psi(G)$ whenever $G\in \bG_n$ and $\pi\in\bS_n$ for some $n\in\bN$.

In analogy to the pattern-based independence tests in~\cite{BaGrEJS3} our tests will use subgraph counts.   
A graph $H=(V(H),E(H))$ is a \emph{subgraph} of $G$ if $V(H)\subset V(G)$ and $E(H)\subset E(G)$, 
so that the subgraph relation is edge-preserving;  we abbreviate this to $H\subset G$. 
Let $I(k,n)$ be the set of injective (one-to-one) functions $\phi:[k]\to [n]$, $n\ge k$, which we may write as
\begin{equation*}
 I(k,n)\, = \, \bigl\{ (i_1,\ldots,i_k)\in [n]:\, i_j\not=i_l\text{ if } j\not=l  \bigr\}.
\end{equation*}
We have $\# I(k,n) = (n)_k:= n(n-1)\cdots (n-k+1)$. 
For $\phi\in I(k,n)$ and $H\in \bG_k$ let $\phi(H)$ be the graph with vertices $\phi(i)$, $i\in [k]$, 
and edges $ \{\phi(i),\phi(j)\}$, $\{i,j\}\in E(H)$. Given $G\in\bG_n$ and $H\in\bG_k$, $k\le n$, let
\begin{equation*}
    t(H,G):= \frac{1}{(n)_k}\,  A(H,G),  \text{ with } A(H,G):= \# \{\phi\in I(k,n):\, \phi(H)\subset G\}.
\end{equation*}
Occasionally we write $A_H$ for the function $G\mapsto A(H,G)$, $G\in\bG$. 
Note that $A(H,G)$ counts the occurrences of $H$ in $G$ with multiplicity $\aut(H)$.
We may interpret $t(H,G)$ as the probability that, when sampling $k$ vertices from $G$ without replacement
by choosing some $\phi\in I(k,n)$ uniformly at random,   
we obtain $\phi(H)\subset G$.  Clearly, these are graph functionals. 
For example, $t(K_2,C_4)= 2/3$. 
In fact, we always have $t(K_2,G) = 2 e(G)/(v(G) (v(G)-1))$, which is the edge density  of $G$. 

Finally, we recall that a simple graph $G=(V,E)$ may be described by its adjacency matrix 
$M=M(G)\in \{0,1\}^{v(G)\times v(G)}$, which has entries 1 in row $i$ and column $j$ if $\{i,j\}\in E$, and 0 otherwise.

\subsection{Random graphs} \label{subsec:graphon}
A \emph{graphon} is a measurable and symmetric function $W:[0,1]^2\to [0,1]$. By the
\emph{graphon construction} we mean a specific procedure that generates a sequence 
$(G_n)_{n\in\bN}$ of random graphs, with $G_n\in\bG_n$ for all $n\in\bN$. For this,  let $U_i$, $i\in \bN$,
and $Y_e$, $e\in E_\infty  := \{\{i,j\}\in\bN:\, 1\le i < j<\infty\}$, be independent random variables, all
uniformly distributed on the unit interval. 
Then the (random) edge sets of the graphs $G_n$ are defined by
$E(G_n) := \bigl\{\{i,j\}\in E_\infty:\, 1\le i <j\le n,\, Y_{\{i,j\}}\le W(U_i,U_j)\bigr\}$, $n\in\bN$.
The parametrization of the distribution of $(G_n)_{n\in\bN}$ 
by $W$ is not injective: Two graphons $W_1$ and $W_2$ 
that differ only on a set of measure zero (with respect to the uniform distribution on the unit square) 
will lead to the same distribution. This also holds  if $W_2(x,y)=W_1(\psi(x),\psi(y)))$ with some measurable
$\psi:[0,1]\to[0,1]$ that leaves the uniform distribution invariant. We will mainly ignore this and write 
$\cW$ for the set of graphons.
Note that there are two steps, and
that we could as well use a graphon $W$ and a sequence $(u_i)_{i\in\bN}$ of pairwise 
different elements of the open unit interval as parameters and then proceed to regard the above 
as a probabilistic mixture of this construction.  


A central result is the following: 
If $(G_n)_{n\in\bN}$ is a sequence of graphs generated from~$W$, then all sampling probabilities
converge almost surely, and the respective limits can be expressed in terms of $W\!$: \ 
For all $k\in\bN$ and $H\in\bG_k$,
\begin{equation}\label{eq:Wconv}
	t(H,G_n)\toas t(H,W)\; 
	:=\ \int_0^1\!\!\cdots\!\! \int_0^1\prod_{\{u_i,u_j\}\in E(H)}W(u_i,u_j)\, du_1\cdots du_k
\end{equation}
as $n\to\infty$. The random variables are non-negative and bounded by 1, hence the convergence also 
holds for the respective moments.

If the graphon is constant and equal to $p$ for some $p\in[0,1]$, then we obtain
the  Erd{\H o}s-R\'enyi (or Gilbert or binomial) random graphs, where the potential edges are independently 
included with probability $p$. We abbreviate this to $G_n\sim \ER(n,p)$. In the Erd{\H o}s-R\'enyi case, i.e.\ for constant graphons, the $U$-variables are not needed in the above graphon construction. 
This has the obvious interpretation as a graph variant of homogeneity. If $W$ only takes the values 0 and 1 
then the graphs are functions of the $U$-variables alone, and with a graphon that is piecewise constant 
on a partition of the unit square into rectangles we obtain the widely used stochastic block models;
see also Section~\ref{sec:altpar} below.

\subsection{Quasirandomness} \label{subsec:quasi}
In the Erd{\H o}s-R\'enyi case, the general formula \eqref{eq:Wconv} leads to the limits $p$ 
for $K_2$ and $p^4$ for $C_4$ so that
\begin{equation}\label{eq:ERequ}
	t(C_4,W)=t(K_2,W)^4 \quad\text{ if }\ W\equiv p\ \text{ a.s. for some } p\in[0,1]. 
\end{equation}
In terms of dense graph limits a famous result of Chung, Graham and Wilson~\cite{CGW} states that 
\begin{equation}\label{eq:ERinequ}
	t(C_4,W)\ge t(K_2,W)^4 \ \text{ for all graphons }\ W,
\end{equation}
with equality only in the case $W\equiv p\,$ a.s.\ for some $p\in[0,1]$. 
This result thus provides a characterization of the  Erd{\H o}s-R\'enyi class, or of the almost surely constant
graphons in $\cW$ via iterated integrals.
In view of its importance for the present paper we give the first step of a proof.
Starting with the definition of $t(H,W)$ 
in~\eqref{eq:Wconv} and using the Cauchy-Schwarz inequality we obtain
\begin{align*}
	t(C_4,W) \, &=  \int_0^1\int_0^1\int_0^1\int_0^1 W(u,v)W(v,w)W(w,x)W(x,u)
	                               \, du\, dv\, dw \, dx\\
	                     &= \int_0^1\!\!\int_0^1\Bigl(\int_0^1 W(u,x)W(w,x)\, du_4\Bigr)
	                                                          \Bigl(\int_0^1 W(u,y)W(w,y)\, du_2\Bigr)
	                                                          \, du\, dw \\
	                      &=     \int_0^1\int_0^1\biggl(\int_0^1 W(u_1,v)W(u_3,v)\, dv\biggr)^2 \, du_1\, du_3 \\ 
	                      &\ge \,    \biggl( \int_0^1\int_0^1\int_0^1 W(u_1,v)W(u_3,v)\, dv
	                                                          \, du_1\,  du_3\biggr)^2\quad
	                       =\ \ t(P_3,W)^2,                                
\end{align*}
where $P_3\in\bG_3$ is the path of length two, with edges $E(P_3)=\{\{1,2\},\{2,3\}\}$. Similarly,
\begin{align*}
	t(P_3,W) \, 
	                     &=\, \int_0^1\int_0^1\int_0^1 W(u,v) W(v,w)\, du\, dv\, dw \\
	                      &=\,  \int_0^1\biggl(\int_0^1 W(u,v)\, du\biggr)^2 \, dv \\ 
	                      &\ge \   \biggl( \int_0^1 W(u,v)\,du\, dv\biggr)^2\quad
	                       =\ \ t(K_2,W)^2,                                
\end{align*}
Taken together this gives~\eqref{eq:ERinequ}. Further, if equality holds 
in~\eqref{eq:ERinequ} then the conditions for equality in the
Cauchy-Schwarz inequality can be used to show that $W$ is almost surely constant. 


\subsection{The algebra of graph functionals}\label{subsec:decomp}

Any graph  $G\in\bG$ may be described, up to isomorphism, by subgraph counting, i.e.\! by the functions $H\mapsto A(H,G)$, $H\in\bG$.
This coding of graphs leads to a description of the linear space of graph functionals,
a space that is even an algebra, where multiplication refers to the pointwise multiplication of functions. 
This implies that any product of graph functionals can be written as a linear combination of subgraph counts. 
For example,
\begin{equation}\label{eq:p3k2}
A_{P_3}\cdot A_{K_2} \; = \, A_{P_3+K_2} + 4A_{P_4}+2A_{S_3}+2A_{K_3}+4A_{P_3}
\end{equation}  
(see also the proofs given in the next section).

For $e\in E_\infty$ let $1_e:\bG\to \{0,1\}$ be the associated edge indicator,
so that $1_e(G)=1$ if $e\in E(G)$ and $1_e(G)=0$ otherwise. Clearly, $H$ is a subgraph of  $G$ 
if and only if $1_e(G)=1$ for all $e\in E(H)$, which can be expressed as $\prod_{e\in E(H)}1_e(G)=1$.
The counts $A(H,G_n)$ may then be written as an iterated sum of this function with respect to
the uniform distribution on $[n]$, provided we exclude integer vectors with repeated entries,
\begin{equation*}
    A(H,G) \ =\ \sum\;\sum\;\dots\;\sum_{\hspace{-1.9cm}(i_1,\ldots,i_k)\in I(n,k)}\quad  \prod_{\{j,l\}\in E(H)}  1_{\{i_j,i_l\}}(G),
\end{equation*}
with $k=v(H)\le n$. Through its analogy to the iterated integral in~\eqref{eq:Wconv} this also points to the 
familiar passage to empirical distributions in the linear (single integral) case.

Writing $1_e=(1_e-p) + p\,$ and expanding we are led to $Z_H:\bG\to \bR$,  
\begin{equation}\label{eq:defZ}
     Z_H(G) \; := \sum_{\phi\in I(k,n)}\, \prod_{e\in\phi(H)} \bigl(1_e(G) - p\bigr), \ \text{ where }
                                      k=v(H) \text{ and }  n=v(G).
\end{equation} 
We augment this with  $Z_\emptyset \equiv 1$.
We will usually regard $p$ as fixed and being clear from the context. Note that $Z_H(G)=A(H,G)$ 
if $p=0$, and that each $Z_H$ is a graph functional. Let $\cH$ be a set of graphs that consists of
the empty graph and, for each $H\in\bG$ without isolated vertices, an element of the isomorphism 
class associated with $H$ (alternatively, we could use unlabeled graphs as indices). 
Any graph functional $\Psi:\bG\to \bR$ then has a unique decomposition
\begin{equation}\label{eq:GFdec}
\Psi \, =\, \sum_{H\in\cH} c(\Psi,H ) \, Z_{H }
\end{equation}
with real coefficients $c(\Psi,H)$, see~\cite[Lemma 6.42]{JLR}. 
The $Z$-variables thus provide a basis for the algebra of graph functionals.  In particular, there is an extension 
to $p\not=0$ of the representation of products as linear combinations. For example, on $\bG_n$ with $n\ge 5$, the extension of~\eqref{eq:p3k2} is
\begin{equation}\label{eq:Zp3k2}
Z_{P_3} Z_{K_2} \  = \ Z_{P_3+K_2} + 4Z_{P_4} +2Z_{S_3}+2Z_{K_3}+4(1-p)Z_{P_3}
                       +4p(1-p)(n-2)Z_{K_2}.
\end{equation}  
Note that, in contrast to the pure counting case with $p=0$, graphs with only two nodes now appear in the 
linear representation of the product.   

Consider, for fixed $n\in \bN$ and $p$, the discrete probability space $(\bG_n,\ER(n,p))$. 
The probabilities of the  individual graphs $G$ are given  $p^{e(G)}(1-p)^{\binom{n}{2}- e(G)}$, and the edge indicators $1_e:\Omega\to\bR$, $e\in E_\infty$, are independent random variables with distribution $\Ber(1,p)$. 
This  implies that the random variables  $Z_{H}$ are uncorrelated for non-isomorphic $H$'s, with mean 0 
(if $H\not=\emptyset$) and variance
\begin{equation}\label{eq:decompL2} 
                \var(Z_H(G_n))=EZ^2_H(G_n)= \aut(H)(n)_{v(H)}(p(1-p))^{e(H)},
\end{equation}
see~\cite[Lemma 6.41]{JLR}. Hence the $Z$-basis is orthogonal in the Hilbert space $L^2(\bG_n,\ER(n,p))$,
and the basis elements have norm $\|Z_H\|_2$ of order $n^{v(H)/2}$ in this space.
This can be used to identify an asymptotically dominant term in an expansion such as~\eqref{eq:GFdec}. 

Moreover, similar to the Hoeffding decomposition in the classical case for functions that are invariant under permutations, the joint distributional behavior of the terms arising in the decomposition~\eqref{eq:GFdec} 
of a graph functional is known. 
The full statement and its proof are given in~\cite[Chapter 11]{JansonGHS}. Here we only require 
that for connected graphs $H\not=\emptyset$ the variables $n^{-v(H)/2}Z_H(G_n)$ converge in distribution
to a centered normal with variance $\aut(H)((p(1-p))^{e(H)}$, and that they are asymptotically independent 
for non-isomorphic graphs $H$;  see also~\cite[Theorem 6.43 and Example 6.50]{JLR}.

\section{Main results}\label{sec:results}
With the construction in~Section~\ref{subsec:graphon} any graphon $W$ 
leads to a distribution $\bP_{W,n}$ on $\bG_n$ for all $n\in\bN$. We now assume that the 
observed graph with $n$ nodes (the data) is the realization of a random graph $G_n$ with distribution $\bP_{W,n}$, 
where $W\in\cW$ is unknown. Note the analogy to the classical case where we observe a sample of size~$n$, 
which we regard as the observation of a random vector with (product) distribution $\mu^{\otimes n}$ where $\mu$ is unknown. 
Our aim is to construct tests of the (composite) hypothesis 
\begin{equation}\label{eq:hyp}
	H_0: \, \cL(G_n) =  \ER(n,p)\  \text{ for some } p\in (0,1).
\end{equation}
We exclude the trivial cases $p=0$ and $p=1$.
Our results are asymptotic in nature, which means that we consider an associated  sequence
$(T_n)_{n\in\bN}$ of tests, where $T_n$ rejects $H_0$ if its value exceeds a critical 
value $c=c(n,\alpha)$ that depends on the chosen level $\alpha$. We want the tests to be asymptotically 
similar, meaning that 
\begin{equation}\label{eq:similar}
    \lim_{n\to\infty}P\bigl(T_n < c(n,\alpha)\bigr)\, =\, \alpha \ \text{ if $H_0\;$ is satisfied}.
\end{equation}
Ideally, they are also consistent in the sense that 
\begin{equation}\label{eq:consistent}
    \lim_{n\to\infty}P\bigl(T_n\ge c(n,\alpha)\bigr)\, =\, 1 \ \text{ otherwise}.
\end{equation}
For \eqref{eq:similar} we need the distribution of $T_n$ under $\ER(n,p)$,
or its limit as $n\to\infty$. To obtain a distributional limit one often applies an affine
transformation, such as subtracting the mean and 
dividing by the standard deviation. This, however, may well involve 
the unknown parameter $p$.  A classical solution is `studentization' 
(here essentially an application of Slutsky's lemma), which works if only the scaling is affected.
If the shifts also depend on $p$ then simply replacing $p$ by a consistent estimator $\hat p_n$ may not be
enough. In the present paper we approach this
problem by using the decomposition outlined in Section~\ref{subsec:decomp} together with the fact the 
space of graph functionals is an algebra, as the product of functions that are
constant on orbits is again constant on orbits. We illustrate this with several results, each providing an asymptotically similar test sequence, 
where the third is also consistent. All test statistics  are (nonlinear) combinations of subgraph counts.

The first test is based on the inequality $t(P_3,W)\ge t(K_2,W)^2$, with equality on the hypothesis; see
Section~\ref{subsec:quasi}. In particular, 
if we select three different nodes $i,j,k\in[n]$ uniformly at random, then the  events $A$ and $B$ that the two incident edges $\{1,2\}$ and $\{2,3\}$ are both in $E(G_n)$ are (weakly) positively dependent in the sense that 
$P(A\cap B)\ge P(A)P(B)$, and independent on $H_0$. From Section~\ref{subsec:decomp} 
we know the distributional asymptotics of the empirical counterpart $t(P_3,G_n)$,
for $t(K_2,G_n)^2$ we could use the known result for $t(K_2,G_n)$ together with the delta method.
However, for the behavior of the difference their joint distribution is needed. It is here that the 
representation in terms of a fixed basis becomes important. 

\begin{theorem}\label{thm:mainP3}
Let $\alpha$, $0<\alpha<1$, be given and let $u_{\alpha}$ be the upper $\alpha$-quantile of the
standard normal distribution. Then the test sequence with test statistics $T_{1,n}$, 
\begin{equation}\label{eq:defT1}
T_{1,n}\, := \, \frac{n^{3/2}}{\sqrt{2} \; t(K_2,G_n)(1-t(K_2,G_n))}\Bigl(t(P_3,G_n)-t(K_2,G_n)^2\Bigr),
            \ \  n\in\bN,
\end{equation}
and critical value $u_{\alpha}$ is asymptotically similar at level $\alpha$. 
\end{theorem}

\begin{proof}
In the first step we relate the subgraph counts to the $Z$-variables.
For the path $P_3\in\bG_3$ with edges $e_1,e_2$ and $G\in\bG_n$ we get 
\begin{align}\label{eq:decompA}
   A(&P_3,G)\ = \ \sum_{\phi\in I(3,n)}  1_{\phi(e_1)}(G)\, 1_{\phi(e_2)}(G)  \notag\\
                          &=\sum_{\phi\in I(3,n)} \bigl((1_{\phi(e_1)}-p)(1_{\phi(e_2)}-p)
                                              + p(1_{\phi(e_1)}-p) + p(1_{\phi(e_2)}-p)  + p^2\bigr)(G)\\
                          &=\quad \ Z_{P_3}(G) \; +\;  2(n-2)\,p\, Z_{K_2}(G) \; + \; n(n-1)(n-2)\, p^2.\notag
\end{align}
Next we represent the graph functional $G\mapsto Z_{K_2}(G)^2$ as a 
linear combination of $Z$-variables.
Clearly, on $\bG_n$, and with $e=\{1,2\}$ the single edge of $K_2$,
\begin{equation*}
  Z^2_{K_2}\  =\  \sum_{(\phi_1,\phi_2)\in I(2,n)\times I(2,n)} \ 
                                   \bigl(1_{\phi_1(e)}-p\bigr) \bigl(1_{\phi_2(e)}-p\bigr).
\end{equation*} 
We can combine $\phi_1$ and $\phi_2$ into an element of $I(k,n)$ for some $k\in\{2,3,4\}$, depending 
on the overlap of the respective range.  For $k=2$ this results in the same factor appearing twice, which 
we handle via $(1_e-p)^2 =  (1-2p)\bigl(1_e-p\bigr) +p(1-p)$. For $k=3$ we obtain $P_3$, for $k=4$ the 
result is $K_2+K_2$. Hence, on $\bG_n$ and observing multiplicities,
\begin{equation}\label{eq:Z2decomp}
  Z_{K_2}^2 \ = \   Z_{K_2+K_2}+ 4\, Z_{P_3} + 
                          2(1-2p)\, Z_{K_2} + 2n(n-1) \, p(1-p),
\end{equation}
which is the desired representation. 

Now we assume that $G_n$ is a random graph with distribution $\ER(n,p)$. In order to derive the 
distributional asymptotics of $t(P_3,G_n)$ we subtract the mean and use~\eqref{eq:decompA},
\begin{align*}
      t(P_3,G_n)- E\,t(P_3,G_n) \ &= \ \frac{1}{n(n-1)(n-2)}A(P_3,G_n) \; - \; p^2\\
       &=\ \frac{1}{n(n-1)(n-2)} Z_{P_3}(G_n) \, + \, \frac{2p}{n(n-1)} Z_{K_2}(G_n).
\end{align*}
Further, with~\eqref{eq:Z2decomp} we get
\begin{equation}\label{eq:A2split}
\begin{split}
  A_{K_2}^2  \ &=\ \bigl( Z_{K_2} +  n(n-1) p\bigr)^2\\
                    &=\ Z_{K_2}^2 + 2n(n-1)p\, Z_{K_2} + n^2(n-1)^2\, p^2\\
                    &=\ Z_{K_2+K_2}\, + \, 4\, Z_{P_3} \, +\, \bigl(2(1-2p)+2n(n-1)p\bigr)\, Z_{K_2}\\ 
                    &\hspace{1.9cm} + \, 2n(n-1)p(1-p)\, +\,  n^2(n-1)^2 p^2.
\end{split}
\end{equation} 
Taken together this leads to the representation
\begin{align*}
  t(P_3,G_n) - t(K_2,G_n)^2 
          &= \     c_1(n) Z_{K_2+K_2}(G_n) \, +\, c_2(n) Z_{P_3}(G_n)\\ 
          &\qquad  +\, c_3(n,p)Z_{K_2}(G_n) \,+\, c_4(n,p),
\end{align*}
where the coefficients are given by
\begin{align*}
        c_1(n) \ &=\    -\frac{1}{n^2(n-1)^2}    ,\\        
        c_2(n) \ &=\    \frac{1}{n(n-1)(n-2)}  -  \frac{4}{n^2(n-1)^2}   ,\\        
        c_3(n,p) \ &=\    \frac{2p}{n(n-1)}  -  \frac{2(1-2p)+2n(n-1)p}{n^2(n-1)^2}  \ 
                              =\ -\frac{2(1-2p)}{n^2(n-1)^2}       ,\\        
        c_4(n,p) \ &=\    p^2  -  \frac{2p(1-p)}{n(n-1)}  - p^2\ =\ -  \frac{2p(1-p)}{n(n-1)} .
\end{align*}
We now use~\eqref{eq:decompL2} to evaluate the order of magnitude of the individual terms. This leads to
\begin{equation}\label{eq:distP2toK2square}
     n^{3/2}\bigl(t(P_3,G_n) - t(K_2,G_n)^2 \bigr) \; - \; n^{3/2} \, Z_{P_3}(G_n) \ = \ o_P(1) \quad\text{as } n\to \infty.
\end{equation}
Note that a cancellation occurred in the $K_2$-term, which had been dominant 
in the separate counts of $K_2$ and $P_3$. The distributional
asymptotics of the statistic $t(P_3,G_n) - t(K_2,G_n)^2$  is thus related to the term associated with paths of 
length two. For $Z_{P_3}$ we can apply~\cite[Theorem 6.45 and Corollary  6.47]{JLR} together with 
\eqref{eq:decompL2} to obtain asymptotic normality
of $n^{3/2}Z_{P_3}$, with mean 0 and variance $\sigma^2(p)=2p^2(1-p)^2$ for the limit distribution. 

It follows from~\eqref{eq:Wconv} that $t(K_2,G_n)$ is a consistent estimator for the unknown $p$. 
Inserting $t(K_2,G_n)$ for $p$ in the formula for the limit variance and using Slutsky's lemma together 
with the continuos mapping theorem we obtain the assertion.
\end{proof}


A test based on triangle counts has been introduced and discussed in~\cite{BFJ}.  Using similar arguments 
as in the proof of Theorem~\ref{thm:mainP3}  we obtain the following result. 
The limit variance now consists of two components, contributed by $K_3$ and $P_3$ respectively. 

\begin{theorem}\label{thm:mainK3}
Let $\alpha$, $0<\alpha<1$, be given and let $u_{\alpha/2}$ be the upper $\alpha/2$-quantile of the
standard normal distribution. Let 
\begin{equation*}
   \phi(G_n) \; := \; 6 \, t(K_2,G_n)^3 \bigl(1-t(K_2,G_n)\bigr)^3 \, +\;
                                18 \, t(K_2,G_n)^4 \bigl(1-t(K_2,G_n)\bigr)^2. 
\end{equation*}
Then the test sequence with test statistics $|T_{2,n}|$, where 
\begin{equation}\label{eq:defT2}
         T_{2,n}\, := \, \frac{n^{3/2}}{\sqrt{\phi(G_n)}}\Bigl(t(K_3,G_n)-t(K_2,G_n)^3\Bigr),    
                         \  \  n\in\bN,
\end{equation}
and critical value $u_{\alpha/2}\,$ is asymptotically similar at level $\alpha$. 
\end{theorem}

\begin{proof}
We follow the steps of the proof of Theorem~\ref{thm:mainP3}. Let $e_1,e_2,e_3$ be the three edges 
of the triangle $K_3\in\bG_3$. The decomposition $1_e=(1_e-p)+p$ leads to
\begin{equation*}
   A_{K_3} \; = \; Z_{K_3}  + 3 pZ_{P_3}  + 3(n-2)p^2Z_{K_2} + n(n-1)(n-2)p^3,
\end{equation*}
on $\bG_n$ and \eqref{eq:decompL2} gives
\begin{equation*}
    \var(Z_{K_3}(G_n)) = 6 (n)_3 (p(1-p))^3, \quad  \var(Z_{P_3}(G_n)) = 2 (n)_3 (p(1-p))^2. 
\end{equation*}
From~\cite[Theorem 6.43 (i) and (ii)]{JLR} it 
follows that $n^{-3/2}Z_{K_3}$ and $n^{-3/2} Z_{P_3}$ are both asymptotically normal and that they are asymptotically independent.  It remains to show that subtracting $t(K_2,G_n)^3$ makes these the
dominant terms in the sense that the coefficients of all other $Z$-variables are of order $o(n^{-3/2})$. 

For the constant terms we can use the general fact that these are always equal to the respective expected 
value. For  $t(K_3,G_n)$ this is $p^3$, for $t(K_2,G_n)^3$ it is $8E(X_n^3)/(n^3(n-1)^3)$,
where $X_n$ has a binomial distribution with parameters $n(n-1)/2$ and $p$. A straightforward 
calculation shows that the difference is of order $n^{-2}$ and thus asymptotically irrelevant if the 
scaling from the assertion is applied.  

For the other terms  we use  $A_{K_2}^3=A_{K_2}^2 A_{K_2}$.
The individual representations are as follows, see~\eqref{eq:Z2decomp} and~\eqref{eq:A2split},
\begin{align*}
    A_{K_2} \; &=\;  Z_{K_2} \,+\, a(n,p),\\
    A_{K_2}^2 \; &=\;   Z_{K_2+K_2}\, +4Z_{P_3} \,+\, b(n,p)Z_{K_2}\, +\, c(n,p),\\
    Z_{K_2}^2 \; &=\; Z_{K_2+K_2}\, +4Z_{P_3} \,+\, d(p)Z_{K_2}\, +\, e(n,p),
\end{align*}
%
with
%
\begin{align*}
     a(n,p) \ &:=\ n(n-1)p,        \\
     b(n,p) \ &:=\ 2(1-2p)\, +\, 2n(n-1)p,          \\
     c(n,p) \ &:=\ 2n(n-1)p(1-p)\, +\, n^2(n-1)^2p^2,          \\     
     d(p) \ &:=\ 2(1-2p),          \\
     e(n,p) \ &:= \ 2n(n-1)p(1-p).     
\end{align*}
For $Z_{K_2}$ the coefficient in $A_{K_2}^3$ evaluates to
\begin{equation*}
 a(n,p) b(n,p) + c(n,p) + b(n,p)d(p)\	= \  3n^2(n-1)^2p^2\; +\; O(n^2),
\end{equation*}
which shows that, asymptotically,  the $Z_{K_2}$-coefficients  cancel 
in the final representation, as in Theorem~\ref{thm:mainP3}. 
(We would obtain $a(n,p)c(n,p)+b(n,p)e(n,p)$ for the constant term, which leads to a computational 
alternative to the above argument involving expectations.)

The remaining terms are
\begin{equation}\label{eq:decompT21}
    a(n,p)\bigl(A_{K_2}^2-b(n,p)Z_{K_2}-c(n,p)\bigr)\; =\; n(n-1)p\bigl(Z_{K_2+K_2}+4Z_{P_3}\bigr)
\end{equation}
and
\begin{equation}\label{eq:decompT21B}
   Z_{K_2}\cdot Z_{K_2+K_2}\, +\, 4\,Z_{K_2}\cdot Z_{P_3} +b(n,p)Z_{K_2+K_2}+4\,b(n,p)Z_{P_3}.
\end{equation}
For these we use~\eqref{eq:decompL2}, resulting in 
\begin{equation*}
    \| n(n-1)pZ_{K_2+K_2}\|_2 =O(n^4), \quad \| 4n(n-1)pZ_{P_3}\|_2 =O(n^{7/2}),
\end{equation*} 
or we use the supremum norm together with the general inequality $\|X\cdot Y\|_2\le \|X\|_\infty\|Y\|_2$ 
to obtain, on $\bG_n$;
\begin{align}\label{eq:infL2}
      \|Z_{K_2}\cdot Z_{K_2+K_2}\|_2 \,&\le  \, \|Z_{K_2}\|_\infty \, \|Z_{K_2+K_2}\|_2 
                                            \, \le \, n(n-1)\|Z_{K_2+K_2}\|_2\, = \, O(n^{4}),\notag \\
      \|Z_{K_2}\cdot Z_{P_3}\|_2 \, &\le  \, \|Z_{K_2}\|_\infty \, \|Z_{P_3}\|_2 \,  
                                            \le \, n(n-1)\|Z_{P_3}\|_2\, = \, O(n^{7/2}),
\end{align}
so that we have the desired rate after multiplication by $n^{-6}$. 

Finally we apply studentization as in the proof of Theorem~\ref{thm:mainP3}. 
\end{proof}

This may seem to be at odds with the result in~\cite[Example 5]{BFJ}, which states that for a valid test 
one would need that the parameter $p$ depends on $n$ such that $np^2\to 0$. 
However, the latter refers to the rate $n$ rather 
than $n^{3/2}$ as in our result, and indeed, the argument in~\cite{BFJ} shows that 
$n(t(K_3,G_n)-t(K_2,G_n)^3)=o_P(1)$ if $p$ is held constant. 
For the standard statistical setup with fixed $p$ a cancellation effect appears
in the transition to the faster rate.

\begin{remark}[Sidedness]\label{rem:sided}
The tests in Theorem~\ref{thm:mainP3}, and Theorems~\ref{thm:mainC4} and \ref{thm:mainC4P3} below, 
are one-sided, whereas 
Theorem~\ref{thm:mainK3} deals with a two-sided test. Note that the critical value is smaller in the
one-sided case. The choice is motivated by the behavior on alternatives (see also Section~\ref{sec:altpar}): 
For $C_4$ we obtain consistency with the one-sided version, and the general inequalities at the end
of Section~\ref{subsec:quasi} can be used to show that in the situation of Theorem~\ref{thm:mainP3}
all non-detectable alternatives $W$ must have $t(P_3,W)=t(K_2,W)^2$. In connection with triangles,
however, non-trivial deviations in both directions may occur. Bipartite graphs provide a classical example: 
There are no triangles at all, whereas the edge density is of the order $1/4$.  With the
one-sided test based on $T_{2,n}$ exceeding $u_\alpha$ this obviously non-homogeneous alternative
would not be detected.  
\brend
\end{remark}

In our next result we obtain a test sequence that asymptotically detects all graphon-based alternatives
with probability 1.
The fact that $T_1$ and $T_2$ do not have this general property will be discussed in Section~\ref{sec:altpar}.

\begin{theorem}\label{thm:mainC4}
Let $\alpha$, $0<\alpha<1$, be given and let $u_{\alpha}$ be the upper $\alpha$-quantile of the
standard normal distribution. Then the test sequence with test statistics
\begin{equation}\label{eq:defT3}
T_{3,n}\, := \, \frac{n^{3/2}}{4\sqrt{2} \; t(K_2,G_n)^3(1-t(K_2,G_n))}\Bigl(t(C_4,G_n)-t(K_2,G_n)^4\Bigr),
\end{equation}
$n\in\bN$, and critical value $u_{\alpha}$ is asymptotically similar at level $\alpha$. 

Moreover, the test sequence is consistent in the sense of~\eqref{eq:consistent}.
\end{theorem}

\begin{proof} 
As in the previous proofs we obtain, on $\bG_n$,
\begin{equation}
\begin{split}
   A_{C_4} \;  = \;  &Z_{C_4} + 4p Z_{P_4} + 2p^2 Z_{K_2+K_2}
                                 + 4(n-3)p^2 Z_{P_3}\\
                                  &\hspace{1cm}+ 4(n-2)(n-3)p^3 Z_{K_2}
                                 + n(n-1)(n-2)(n-3)p^4.
\end{split}
\label{eq:decompC4}
\end{equation}	
which leads to a representation of $t(C_4,G_n)$.
To obtain a corresponding representation for $t(K_2,G_n)^4$ we use the formula for $A_{K_2}^2$ from the
proof of Theorem~\ref{thm:mainK3}. Taking the square leads to ten different terms. Of these,
$Z_{K_2+K_2}^2$, $Z_{P_3}Z_{K_2+K_2}$, $Z_{P_3}^2$, $b(n,p)Z_{K_2}Z_{K_2+K_2}$, and
$b(n,p) Z_{K_2}Z_{P_3}$ are easily handled through the $\|\cdot\|_\infty$ and  $\|\cdot\|_2$ 
norms as in the previous proof. For $b(n,p)^2 Z_{K_2}^2$ we similarly use the representation 
of $Z_{K_2}^2$ given there. The remaining part of the square is
\begin{equation*}
c(n,p)^2 + 2 c(n,p)b(n,p)Z_{K_2} + 8 c(n,p)Z_{P_3}.
\end{equation*}
Dividing this by $(n)_2^4$, letting $n\to \infty$ and ignoring terms of smaller order than $n^{-3/2}$
this reduces to $p^4+4n^{-2}p^3Z_{K_2}$, so that we obtain the desired cancellation. 
This leads to
\begin{equation*}
    \bigl(t(C_4,G_n) - t(K_2,G_n)^4\bigr)\; -\; 4n^{-3/2}p^2 \,Z_{P_3}(G_n) \ \toP \ 0,
\end{equation*}
with the asymptotic normality of $n^{-3/2}Z_{P_3}(G_n)$ it follows that
\begin{equation*}\label{eq:T0}
                n^{3/2} \bigl(t(C_4,G_n)-t(K_2,G_n)^4\bigr)\ \todistr \ N\bigl(0, 32p^6(1-p)^2\bigr),
\end{equation*}
and we use studentization  to obtain~\eqref{eq:defT3}.

Finally, for consistency, a standard argument applies: 
Assume that $W$ is not almost surely equal to a constant. 
By the quasirandomness result in Section~\ref{subsec:quasi} we then have 
\begin{equation*}
         \lim_{n\to\infty}t(C_4,G_n)=t(C_4,W) > t(K_2,W)^4 =\lim_{n\to \infty} t(K_2,G_n),
\end{equation*}
all with probability~1.
Thus, $T_{3,n}\to \infty$ a.s., which means that the hypothesis will be rejected with probability~1 as the size 
$n$ of $G_n$ tends to $\infty$.
\end{proof}

\begin{remark}[Bias correction]\label{rem:bias}
In the above results 
we used $t(K_2,G_n)$ to estimate $p$ and extended this by the `plug-in 
principle' to $t(K_2,G_n)^j$ to obtain estimators for the powers $p^j$ of $p$,  with $j=2,3,4$. 
We may obviously replace 
these estimators of $p^j$ by any other estimator $\widehat{p_j}(G_n)$ without changing the limit distribution
as long as $\widehat{p_j}(G_n)-t(K_2,G_n)^j=o_P(n^{-3/2})$. A natural choice
is based on the fact that the factorial moments $E(X)_j$ of $X\sim\Bin(k,p)$ are given by $ (k)_j\, p^j$.
Recalling that $A(K_2,G_n)/2\sim \Bin(n(n-1)/2,p)$ under $\ER(n,p)$ this motivates the use of
\begin{equation}\label{eq:unbiased}
   \widehat{p_j}(G_n)\, := \; \frac{\displaystyle\frac{A(K_2,G_n)}{2}
                  \biggl(\frac{A(K_2,G_n)}{2}-1\biggr)\cdot\ldots \cdot\biggl(\frac{A(K_2,G_n)}{2}-j+1\biggr)}
        {\displaystyle\frac{n(n-1)}{2}\biggl(\frac{n(n-1)}{2}-1\biggr)\cdot \ldots\cdot\biggl(\frac{n(n-1)}{2}-j+1\biggr)}.
\end{equation}
For $X\sim\Bin(k,p)$ we have $X^j\ge (X)_j$ and $E\bigl(X^j-(X)_j\bigr)\le c_j k^{j-1}$ where the constant 
$c$ does not depend on $k$. Together 
with Chebyshev's inequality this implies that the rate condition is satisfied. Replacing the plug-in estimator by 
$\widehat{p_j}(G_n)$ then leads to tests that are bias corrected in the sense that the new test statistics 
$T_{l,n}\bc$, $l=1,2,3$, have mean 0 on $H_0$. 
\brend
\end{remark}

A motivation for the test $T_1$ in Theorem~\ref{thm:mainP3} was its appearance in one of the two
inequalities in the quasirandomness argument in Section~\ref{subsec:quasi}. The other inequality
similarly leads to a fourth test, now based on the comparison of the counts of $C_4$ and $P_3$. 

\begin{theorem}\label{thm:mainC4P3}
Let $\alpha$, $0<\alpha<1$, be given and let $u_{\alpha}$ be the upper $\alpha$-quantile of the
standard normal distribution. Then the test sequence $T_4=(T_{4,n})_{n\in\bN}$ with 
\begin{equation}\label{eq:defT4}
T_{4,n}\, := \, \frac{n^{3/2}}{2\sqrt{2} \; t(K_2,G_n)^3(1-t(K_2,G_n))}\Bigl(t(C_4,G_n)-t(P_3,G_n)^2\Bigr),
\ n\in\bN,
\end{equation}
and critical value $u_{\alpha}$ is asymptotically similar at level $\alpha$. 
\end{theorem}

\begin{proof} 
As in the previous proofs our strategy is to write $t(C_4,G_n)$ and $t(P_3,G_n)^2$ as linear combinations
of $Z_H$-variables; to order these according to their
magnitude in powers of $n$;  to obtain a cancellation of the first two terms; and finally to identify the
terms of order $n^{-3/2}$. Everything below $n^{-3/2}$ may be ignored. 

We know from the proof of Theorem~\ref{thm:mainC4} that
\begin{equation}
   t(C_4,\,\cdot\,) \;  = \;   
                                 \frac{4p^2}{(n)_3}\, Z_{P_3}
                                 \, +\,  \frac{4p^3}{(n)_2} \, Z_{K_2}
                                 \,+\, p^4  \; + \; o_P(n^{-3/2}).
\end{equation}	
In order to arrive at a similar representation for $t(P_3,G_n)^2$ we use the representation of $A(P_3,G)$
in~\eqref{eq:decompA} to obtain 
\begin{equation}
   t(P_3,\,\cdot\, )\; =\; \frac{1}{(n)_3}\,  Z_{P_3} \, +\,  \frac{2p}{(n)_2}\, Z_{K_2} \; + p^2.
\end{equation}
The terms in the square that involve the constant $p^2$ sum to
\begin{equation*}
\frac{2p^2}{(n)_3}\, Z_{P_3}
                                 \, +\,  \frac{4p^3}{(n)_2} \, Z_{K_2}
                                 \,+\, p^4.
\end{equation*}
If all other terms are $o_p(n^{-3/2})$ then it would follow that
\begin{equation*}
      n^{3/2}\bigl(t(C_4,G_n)-t(P_3,G_n)^2\bigr) - 2p^2 n^{3/2} Z_{P_3} \, = \, o_p(1),
\end{equation*}
which would yield the assertion as in the proof of Theorem~\ref{thm:mainC4}.

It remains to consider $Z_{P_3}^2$,  $Z_{P_3}Z_{K_2}$ and $Z_{K_2}^2$. The third of these
has been decomposed in the proof of Theorem~\ref{thm:mainP3}, which leads to
\begin{equation*}
\frac{1}{(n)_2 (n)_2}Z_{K_2}^2\, 
                                    = \ \frac{1}{(n)_2 (n)_2}\Big(Z_{K_2+K_2}+4Z_{P_3} +2p(1-2p)Z_{K_2}\Bigr)
                                           + \frac{2p(1-p)}{(n)_2}.
\end{equation*}
Using the $L^2$-bounds in \eqref{eq:decompL2} this is easily seen to be $o_P(n^{-3/2})$. Further, 
it follows from~\eqref{eq:infL2} that $((n)_3(n)_2)^{-1}Z_{P_3}Z_{K_2} =o_P(n^{-3/2})$.

We do not have a similar representation for the remaining $Z_{P_3}^2$ but instead go through the 
various terms that arise in the sum over $\phi_1\times\phi_2\in I(3,n)\times I(3,n)$, according 
to the size $k$ of the union of the range of $\phi_1$ and the range of $\phi_2$.

Obviously, the greatest value of $k$ to be considered is $k=6$, which arises if the two range sets
are disjoint, contributing $Z_{P_3+P_3}$ with $L_2$-norm $O(n^3)$. In view of the later division 
by $(n)_3(n)_3$ this term may be ignored asymptotically.
Suppose next that the functions $\phi_1$ and $\phi_2$ share exactly one value.
Each of the nine possibilities  results in a graph with five nodes, and we can proceed as before. 
If the range intersection has size two then we either
obtain a new edge (if $\phi_1(1)$ and $\phi_1(3)$ appear in the range of $\phi_2$) 
or an existing edge now appears twice. 
The former leads to the graph $H_1$, informally described as $K_3$ with a handle, with four possibilities,
and can be handled with the $L^2$-bound. The remaining possibilities each have a double edge $e_1$ 
with a (single) edge $e_2$ attached. For these we use 
\begin{equation*}
(1_{e_1}-p)^2 (1_{e_2}-p)\, = \, (1-2p)(1_{e_1}-p)(1_{e_2}-p) + p(1-p)(1_{e_2}-p).
\end{equation*}
This leads to overlap two contributions by $Z_{H_1}$,  $Z_{P_3}$ and  $(n-2)Z_{K_2}$ to the product,
again a case for the $L^2$-bound.  

It remains to consider the situation where all elements of the range of 
$\phi_1$ reappear in the range of $\phi_2$, with two possibilities for the three values
of $\phi_2$. Now we use a similar expansion for $(1_{e_1}-p)^2 (1_{e_2}-p)^2$
to identify this last bit with the linear combination 
$(1-2p)^2 Z_{P_3}+ 2(n-2) p(1-p)Z_{K_2} + (n)_3 p^2(1-p)^2$, 
which is again easily disposed of.
\end{proof}

The results add a distributional aspect to the quasirandomness property. The latter is based on
the equality criterion in the two uses of the Cauchy-Schwarz inequality in Section~\ref{subsec:quasi},
where $t(P_3,W)\ge t(K_2,W)^2$ respectively  $t(C_4,W)\ge t(P_3,W)^2$ are the basis for 
Theorem~\ref{thm:mainP3} and Theorem~\ref{thm:mainC4P3}. Neither of these
characterizes homogeneity; for example equality in the second case only implies that 
the function $u\mapsto \int_0^1 W(u,v)\, dv$ is almost surely constant. For homogeneous graphs
the corresponding fluctuations are related. We can obviously connect the various test statistics
through 
\begin{equation*}
\begin{split}
  t(C_4,G_n)-t(K_2,G_n)^4\ &= \ t(C_4,G_n)-t(P_3,G_n)^2\\
         &+\ \Bigl(t(P_3,G_n)-t(K_2,G_n)^2\Bigr)\Bigl(t(P_3,G_n)+t(K_2,G_n)^2\Bigr),
\end{split}
\end{equation*}
and it follows from the proofs that the leading terms in Theorem~\ref{thm:mainP3},
Theorem~\ref{thm:mainC4} and Theorem~\ref{thm:mainC4P3} are $Z_{P_3}$, $4p^2Z_{P_3}$
 and $2p^2Z_{P_3}$ respectively. The decomposition can also be used to obtain an alternative proof for
Theorem~\ref{thm:mainC4P3}.

\section{Two parametric families}\label{sec:altpar}
The argument given at the end of the proof of Theorem~\ref{thm:mainC4} can be applied
to the tests in the other results. 
For example, if  $W$ is a graphon with $t(P_3,W)>t(K_2,W)^2$ then 
the test $T_1$ in Theorem~\ref{thm:mainP3}  will asymptotically reject the hypothesis, i.e.
\begin{equation}\label{eq:loccons}
\lim_{n\to\infty} P_W\bigl(T_{1,n}\ge c(n,\alpha)\bigr) \, = \, 1.
\end{equation}
Similarly, \eqref{eq:loccons} holds for the test $|T_2|$ in Theorem~\ref{thm:mainK3} whenever
$W$ is such that 
$t(K_3,W)\not=t(K_2,W)^3$, and for the test $T_4$  in Theorem~\ref{thm:mainC4P3} if 
$t(C_4,W)>t(P_3,W)^2$. The converse, of course, need not be true: For a particular alternative $W$
we may have limit 1 for the power of the test even if equality holds in the respective condition.

The notion of consistency refers to the specific family $\cP$ underlying the statistical model. 
Theorem~\ref{thm:mainC4} is based on the full nonparametric family $\cW$
of all graphons. In order to further evaluate the other tests we now consider two popular parametric 
subfamilies; the null model $\cP_0$ continues to be the class of Erd{\H o{s-R\'enyi} graphs.
Thus, the limiting null distributions and the property of being similar are not affected by 
the reduction. The results are qualitative; see Section~\ref{sec:conclusion} 
for remarks on quantitative notions of efficiencies.

We first consider a class of simple stochastic block models,
based on the graphons $W=W(\,\cdot\,|\gamma,p_1,p_2)$ given by
\begin{equation}\label{eq:defSBM}
        W(x,y|\gamma,p_1,p_2) := 
                            \begin{cases}
                                      p_1, &\text{if } x\vee y\le \gamma \text{ or } x\wedge y>\gamma,\\
                                      p_2, &\text{otherwise.}
                            \end{cases}
\end{equation}
Here $0<\gamma<1$ defines the block size, $p_1$ is the within-blocks probability for an edge
(here assumed to be identical for the two blocks), and $p_2$ is the probability for edges between 
the two blocks. Note that we now have two sources of randomness; see the end of Section~\ref{subsec:graphon}. 
With $p_1=p_2$ the homogeneous case reappears. Indeed the family models a specific class
of alternatives that can be chosen to be arbitrarily close to elements of the hypothesis $H_0$ in~\eqref{eq:hyp}. 
Further, the extreme cases $p_1=1$, $p_2=0$, and $p_1=0$,  $p_2=1$, correspond to a graph that
consists of two cliques and a bipartite graph respectively.

\begin{proposition}\label{prop:SBM}
{\rm (a)} The test $T_1$ from Theorem~\ref{thm:mainP3} is consistent for the family $\cP\subset \cW$ 
of all block graphons  $W=W(\,\cdot\,|\gamma,p_1,p_2)$  with block size $\gamma\not=1/2$.

\vspace{.3mm}
{\rm (b)} The test $|T_2|$ from Theorem~\ref{thm:mainK3} is consistent for the family $\cP\subset \cW$ 
of all block graphons  $W=W(\,\cdot\,|\gamma,p_1,p_2)$  with $\gamma=1/2$, or with $\gamma\not=1/2$, 
$p_1\ge p_2$.

\vspace{.3mm}
{\rm (c)} The test $T_4$ from Theorem~\ref{thm:mainC4P3} is consistent for the family $\cP\subset \cW$ 
of all block graphons  $W=W(\,\cdot\,|\gamma,p_1,p_2)$  with block size $\gamma\not=1/2$.

\end{proposition}

\begin{proof} (a) For the graphons $W(\,\cdot\,|\gamma,p_1,p_2)$ it is straightforward to calculate $t(H,W)$ from~\eqref{eq:Wconv} for small $H$ as a function of the parameters
$\gamma,p_1,p_2$,
\begin{align*}
      t(K_2,W)\; &= \; (1-2\gamma+2\gamma^2)\, p_1\, +\, 2\gamma(1-\gamma)\,p_2 ,\\
      t(P_3,W)\; &= \;  (1+3\gamma^2-3\gamma)\,p_1^2 \, + \, 
                                  \gamma(1-\gamma)\,p_2^2\,  +\,2\gamma(1-\gamma) p_1p_2.
\end{align*} 
This implies
\begin{equation}\label{eq:SBMT1}
   t(P_3,W)-t(K_2,W)^2 \, =\,  \gamma(1-\gamma)(1-2\gamma)^2 (p_2-p_1)^2,
\end{equation}
and the statement follows with the general remarks above, as the right hand side is greater than 0 
on the alternatives in the model under consideration.

(b) We now use 
\begin{equation*}
     t(K_3,W)\; =\;   (1-3\gamma+3\gamma^2)\,p_1^3\, 
                                                 +\, 3\gamma(1-\gamma)\,p_1p_2^2,
\end{equation*}
which leads to 
\begin{equation*}
          t(K_3,W)-t(K_2,W)^3 \, =\, \gamma(1-\gamma)(p_1-p_2)^2 f(\gamma,p_1,p_2) 
\end{equation*}
with
\begin{equation*}
        f(\gamma,p_1,p_2) \, :=\, 8\gamma^4(p_1-p_2)-16\gamma^3(p_1-p_2)+4\gamma^2(5p_1-2p_2)
                                                    +(3-12\gamma)p_1.
\end{equation*}
For $\gamma=1/2$ this  simplifies to 
\begin{equation*}
     t(K_3,W) - t(K_2,W)^3 \, = \, \frac{1}{8}(p_1-p_2)^3, 
\end{equation*}
and we may proceed as in the proof of (a), adapting the argument to the two-sided case.  For
$\gamma\not=1/2$ additional roots may appear through $f$. As this function is linear in $p_1$ for $\gamma$ and $p_2$ fixed the condition $f(\gamma,p_1,p_2)=0$ can easily be rewritten as
\begin{equation}\label{eq:ausnahmeT2}
      \frac{p_2}{p_1}\, =\, \frac{8\gamma^4-16\gamma^3+20\gamma^2-12\gamma+3}
                                                     {8\gamma^4(1-\gamma)^2}
                                 \, =\, 1+\frac{3(1-2\gamma)^2}{8\gamma^4(1-\gamma)^2}.
\end{equation}
In view of $\gamma\not=1/2$ this can only happen if $p_1<p_2$.

(c) From the final remark in Section~\ref{subsec:quasi} we know that this follows from the fact
that $u\mapsto\int_0^1 W(u,v|\gamma,p_1,p_2)\, dv$ is not constant if $\gamma\not=1/2$
and $p_1\not=p_2$.  
\end{proof}

We can use the equations in the proof to find graphons that are not almost surely 
constant but nevertheless satisfy $t(P_3,W)-t(K_2,W)^2=0$ or $t(K_3,W)-t(K_2,W)^3 = 0$.
For example, it follows from~\eqref{eq:SBMT1} that the first equality is always satisfied if the blocks 
are of equal size, i.e.\ $\gamma=1/2$.
For the second  specific solutions can be obtained from~\eqref{eq:ausnahmeT2}.
In connection with $T_4$ we note that $t(C_4,W)-t(P_3,W)^2=0$ holds 
if $u\mapsto\int_0^1W(u,v)\, dv$ is almost surely
constant on the unit interval. This is easily seen to be the case with $W(\,\cdot\,|\gamma,p_1,p_2)$
whenever $\gamma=1/2$.

In  the second example we consider the graphons $W(\,\cdot\,|\beta,p)$, $0<\beta\le 1$, $0<p<1$, 
with $W(x,y|\beta,p)= p$ if $|x-y|\le \beta$ and $W(x,y|\beta,p)=0$ otherwise. 
These are also known as (linear) band graphons. 
Note that the homogeneous case now corresponds to $\beta=1$. 
For this family the two steps in the graphon construction are easily separated: Through the
parameter $\beta$ the  first $n$ values of the $U$-sequence determine the set of potential pairs, 
then each of these pairs is randomly connected by an edge, independently and with probability equal to 
the second parameter~$p$. 

\begin{proposition}\label{prop:linBG}
The tests $T_1$, $T_2$ and $T_4$ from Theorems~\ref{thm:mainP3}, \ref{thm:mainK3}, 
\ref{thm:mainC4P3} respectively are all 
consistent for the family $\cP\subset \cW$ of band graphons  $W=W(\,\cdot\,|\beta,p)$, $0<\beta\le 1$,
$0<p<1$.
\end{proposition}

\begin{proof}
For $T_1$ we proceed as in the proof of Proposition~\ref{prop:SBM}. With $W=W(\,\cdot\,|\beta,p)$
we now get 
\begin{align*}
        t(K_2,W) \ &=\ p\, \beta(2-\beta),\\
        t(P_2,W)\ &=\ p^2\begin{cases}
                                 4\beta^2-\frac{10}{3}\beta^3 &\text{ for } 0\le\beta\le 1/2,\\
                                 -\frac{1}{3}+2\beta -\frac{2}{3}\beta^3, &\text{ for } 1/2<\beta\le 1.
      \end{cases}
\end{align*}
For $t(P_3,W)-t(K_2,W)^2=0$, assuming that $p>0$, we need a solution $\beta$ of
\begin{equation*}
4-\frac{22}{3}\beta  + \beta^2 = 0\quad\text{or}\quad  
                -\frac{1}{3}+2\beta-4\beta^2+\frac{10}{3}\beta^3-\beta^4= 0
\end{equation*}
in the interval $[0,\frac{1}{2}]$ (left equation) or in the interval $(\frac{1}{2},1]$ (right equation).
There are none, as the first polynomial factors into $\frac{1}{3}(\beta-6)(3\beta-4)$, 
and the second into $(1-\beta)^3(\beta-\frac{1}{3})$.

Evaluating $t(K_3,W)$ via the multiple integrals in~\eqref{eq:Wconv} is somewhat tedious. Instead we note
that, with $u,v,w$ the values of $U_i,U_j,U_k$, these indices are candidates for a triangle 
if and only if $|u-v|\vee|v-w|\vee |w-u|\le \beta$, which is equivalent to the span of the three $U$-variables 
being less than or equal to $\beta$. The known connection to beta distributions leads to 
$t(K_3,W)  =   p^3 \beta^2(3-2\beta)$, and we can proceed 
as before.

For $T_4$ we use the same argument as in the proof of Proposition~\ref{prop:SBM}.
\end{proof}

For graphons with values $0$ and $1$ only, such as in the stochastic block model with  $p_1=0$, $p_2=1$ 
or $p_1=1$, $p_2=0$, or in the linear band model with $p=1$, the subgraph 
counts can be related to classical $U$-statistics for symmetric kernels. For example, counting 
triangles in the second model leads to
\begin{equation*}
   A(K_3,G_n) = \sum_{(i,j,k)\in I(3,n)} h(U_i,U_j,U_k)
\end{equation*}
 where the kernel $h$ (of order three) is given by
\begin{equation*}
    h(u,v,w)= \begin{cases}
                               1, &\text{ if } |u-v| \vee |v-w| \vee |w-v| \le \beta,\\
                               0, &\text{otherwise}.
                     \end{cases}
\end{equation*}
This is implicit in the order-statistics argument used in the above proof. 

\section{Simulations}\label{sec:sim}
The results in the previous  sections refer to the asymptotic behavior of the tests
as the number of graph vertices grows to infinity. Also,
consistency is an entirely qualitative property. For applications the power (rejection rate)  
of the tests at specific alternatives and the behavior of the tests for finite graphs is of major interest.  
Here we present some related simulation results and leave theoretical  investigations to future research; 
see also Section~\ref{sec:conclusion}.

A direct implementation of the sum over a set $I(k,n)$ for fixed $k$ leads to about $n^k$ individual terms.
On first sight it therefore seems that the use of $C_4$ in $T_3$ comes with a major  computational drawback.
Fortunately, as is well known, the counts $A(H,G)$ can in some cases
be related to powers of  the adjacency matrix  $M=M(G)$ of $G\in \bG_n$. This may reduce
the computational burden considerably. Writing $\Sigma(M)$ for the sum of all entries of $M$ we have
\begin{equation*}
   A(K_2,G)= \Sigma(M),\ A(P_3,G)= \Sigma(M^2)- \trace(M^2),\ A(K_3,G)= \trace(M^3),
\end{equation*} 
and finally that 
\begin{equation*}
   A(C_4,G)\ =\  2\,\sum_{i=1}^{n-1}\sum_{j=i+1}^n (M^2)_{ij}\bigl((M^2)_{ij}-1\bigr).
\end{equation*} 
In particular, the computation of the statistics appearing in Section~\ref{sec:results} can all be done with 
at most two matrix multiplications. For these efficient implementations exist.

\begin{table} 
	\caption{Rejection rates for the stochastic block model,\\ 
                  with $p_1=0.5+\beta$, $p_2=0.5-\beta$, $n=100$ and 10000 repetitions.\\
                  All values rounded to multiples of $10^{-3}$.\\[1mm]
                                                                                  }\label{tab:power}
{
\renewcommand{\arraystretch}{1.02} 
\setlength{\tabcolsep}{6.1pt}     

	\begin{tabular}{ccccccccccccc}	
		$\beta$ && -0.4 &  -0.2 & -0.1 &-0.05 & 0.0 & 0.05 & 0.1  & 0.2 & 0.4 \\  
		\noalign{\vspace{.1mm}}
        \noalign{\hrule}

		\noalign{\vspace{1.3mm}}
		$T_1$  && 0.249 &  0.132 &  0.055 &  0.046 &  0.043 &  0.047 &  0.055 &  0.132 &  0.239 \\
        $T_1\bc$ &&  0.254 &  0.141 &  0.067 &  0.06 &  0.054 &  0.059 &  0.068 &  0.142 &  0.245  \\
		$T_2$     && 0.0 &  0.0 &  0.002 &  0.032 &  0.044 &  0.063 &  0.436 &  1.0 &  1.0 \\ 
		$|T_2|$  && 1.0 &  1.0 &  0.443 &  0.059 &  0.051 &  0.054 &  0.325 &  1.0 &  1.0  \\ 
		$T_3$  && 1.0 &  0.991 &  0.068 &  0.039 &  0.036 &  0.041 &  0.079 &  0.950 &  1.0 \\                        
        $T_3\bc$ && 1.0 &  1.0 &  0.092 &  0.059 &  0.052 &  0.054 &  0.103 &  0.969 &  1.0   \\
         $T_4$  && 1.0 &  1.0 &  0.092 &  0.034 &  0.034 &  0.037 &  0.110 &  1.0 &  1.0\\ 
	\end{tabular}
}
\begin{center}
(a) \textit{The symmetric case $\gamma=0.5$}.
\end{center}
\vspace{2mm}

{
\renewcommand{\arraystretch}{1.02} 
\setlength{\tabcolsep}{6.1pt}     

	\begin{tabular}{ccccccccccccc}	
		$\beta$ && -0.4 &  -0.2 & -0.1 &-0.05 & 0.0 & 0.05 & 0.1  & 0.2 & 0.4 \\  
		\noalign{\vspace{.1mm}}
        \noalign{\hrule}

		\noalign{\vspace{1.3mm}}
		$T_1$  &&0.801 &  0.659 &  0.230 &  0.083 &  0.043 &  0.085 &  0.285 &  0.645 &  0.802  \\
        $T_1\bc$ &&  0.805 &  0.673 &  0.319 &  0.102 &  0.054 &  0.102 &  0.314 &  0.662 &  0.805      \\		
        $T_2$  && 0.0 &  0.010 &  0.049 &  0.054 &  0.044 &  0.107 &  0.680 &  1.0 &  1.0  \\ 
		$|T_2|$  &&  1.0 &  0.966 &  0.257 &  0.062 &  0.051 &  0.078 &  0.592 &  1.0 &  1.0 \\ 
		$T_3$  &&1.0 &  0.999 &  0.272 &  0.070&  0.036 &  0.080 &  0.369 &  0.990 &  1.0 \\                        
        $T_3\bc$ && 1.0 &  0.999 &  0.320 &  0.096 &  0.052 &  0.108 &  0.419 &  0.995 &  1.0  \\
         $T_4$  &&  1.0 &  1.0 &  0.256 &  0.057 &  0.031 &  0.077 &  0.450 &  1.0 &  1.0 \\
	\end{tabular}
}
\begin{center}
(a) \textit{An unymmetric case, with  $\gamma=0.4$}.
\end{center}
\end{table}

In our first simulation experiment we consider the simple stochastic block models
discussed in Section~\ref{sec:altpar}. 
Table~\ref{tab:power} shows the results for simulated data from $W(\,\cdot\,|\gamma,1/2-\beta,1/2+\beta)$
with $\gamma=0.5$ in its first and $\gamma=0.4$ in its second part, both
for several values of $\beta$, with $n=100$ nodes, level $\alpha=0.05$, and 10000 Monte Carlo repetitions. 
The numbers support the following points:
\begin{itemize}
\itemsep=2pt
\item[--]{the false rejection rates
                are at the specified level $\alpha$, up to random fluctuations (see the column with $\beta=0$),}
\item[--]{for $T_2$, one should use the two-sided test (see Remark~\ref{rem:sided}),}
\item[--]{bias correction improves the performance (see Remark~\ref{rem:bias}),}
\item[--]{in cases where the local consistency condition is not satisfied, the rejection rate may still  
                be close to 1 (as in the case of $T_4$), or far from it (as in the case of $T_1$). }
\end{itemize}
Further, comparing the numbers for $T_1$ 
in the symmetric and unsymmetric case supports the results in Proposition~\ref{prop:SBM}.
Overall the two-sided test based on triangles shows the best performance, especially so in
the symmetric case. It is, however, well-known that the omnibus property (consistency against all
alternatives) usually comes at the price of reduced power, specifically for those
alternatives that are tailor-made for a specific test.

\begin{table} 
	\caption{Rejection rates for the band graphon\\ 
                    with band width $\beta$, $p=0.5$, $n=100$, and 10000 repetitions. \\
                    All values rounded to multiples of $10^{-3}$.}\label{tab:powerBand}
{
\renewcommand{\arraystretch}{1.02} 
\setlength{\tabcolsep}{6.1pt}     

	\begin{tabular}{cccccccccc}	
		$\beta$ && 0.05 & 0.1 & 0.3 &  0.5 & 0.7 & 0.9 & 0.95 & 1.0 \\  
		\noalign{\vspace{.1mm}}
        \noalign{\hrule}

		\noalign{\vspace{1.3mm}}
 $T_1$ 	 &&  0.166 &  0.320 &  0.980 &  1.0 &  0.992 &  0.102 &  0.050 &  0.046 \\
 $T_1\bc$ 	 &&  0.182 &  0.342 &  0.983 &  1.0 &  0.994 &  0.124 &  0.064&  0.059\\
 $T_2$ 	 &&  1.0 &  1.0 &  1.0 &  1.0 &  0.991 &  0.093 &  0.048 &  0.045\\
 $|T_2|$ 	 &&  1.0 &  1.0 &  1.0 &  1.0 &  0.985 &  0.072 &  0.053 &  0.053\\
 $T_3$ 	 &&  1.0 &  1.0 &  1.0 &  1.0 &  0.989 &  0.086 &  0.043 &  0.038\\
 $T_3\bc$ 	 &&  1.0 &  1.0 &  1.0 &  1.0 &  0.992 &  0.116 &  0.061 &  0.055\\
 $T_4$ 	 &&  1.0 &  1.0 &  1.0 &  1.0 &  0.982 &  0.072 &  0.034 &  0.032 \\
    \end{tabular}
}
\end{table}
Table~\ref{tab:powerBand} shows the simulation results for the band graphon with $p=1/2$ and various 
$\beta$-values, with  $n$, $\alpha$ and the number of repetitions as in the previous experiment. 
Note that now the right-most column corresponds to the hypothesis. In this model the rejection rates for the
bias-corrected consistent test are  slightly higher than for the test that uses triangles. As it befits an
experiment, the numbers generate new questions, for example the relevance of the two-stage randomness
in connection with the behavior of $T_1$ in the band graphon model for narrow bands.

\section{Conclusion and outlook}\label{sec:conclusion}

In the context of the limit theory for dense graphs the notion of quasirandomness, originally referring to 
deterministic sequences, leads to a characterization of a specific set of limit points; see~\cite{JansonQuasi}
for an in-depth treatment of this connection. Of course, the use of 
characterizations of probability distributions to obtain consistent  goodness-of-fit tests has a long 
tradition in the classical case where the data set is a sample from some distribution. The tests here 
are based on comparing subgraph counts, other quasirandomness results should lead to alternative 
procedures. Also, there are analogues of the limit theory for dense graphs for other discrete structures. 
Binary trees, for example, have been investigated in~\cite{EGW2} along lines parallel to the 
graph/graphon  situation.


The connection to permutations and independence tests has briefly been mentioned in the introduction.
We now expand on this, partly because the analogy may be of interest in its own right, but also because it may 
lead to (statistical) applications for other discrete structures.

Suppose that $(Z_n)_{n\in\bN}$ is a sequence of independent and identically distributed two-dimensional 
random vectors with continuous marginal distribution functions. 
The first $n$ of these define a rank plot which we rescale in 
order to get a discrete subset of the unit square. With $n\to \infty$ we then obtain a copula, which in 
discrete mathematics is known as a permuton. The components of the $Z$-vectors are independent 
if and only if  this limit is the independence copula, or uniform permuton. Reordering the $Z$-values, which we may 
regard as an action of $\bS_n$ on $Z_1,\ldots.Z_n$, leads to the same rank plot. As reordering the sample values 
is statistical irrelevant in this context we pass on to the `unlabeled' rank plot.
In contrast to the graph situation, there is now a canonical choice of representative, for example
via the linear order of one of the component values. Also, while the symmetric group acts on the `raw' 
objects in both cases, the equivalence classes are  all of the same size in the rank situation, whereas
the transition from labeled to unlabeled graphs is considerably more complicated. 

%

Another difference is in the `permuton construction', meaning the construction of a sequence 
$(\Pi_n)_{n\in\bN}$ of random permutations that converges almost surely to a given permuton (or copula)
$C$ in the topology defined by convergence of pattern frequencies:
We sample $Z_1,Z_2,\ldots$ from the distribution with distribution function $C$ and take $\Pi_n$ to be the 
canonical representation of the rank plot associated with the first $n$ of these. Thus only one stage 
of randomness is needed, whereas in the graph situation
there are two. Loosely speaking, the first of these 
selects the position
of the nodes in the limiting adjacency matrix $W$, and the second then decides, based on the first step, 
whether or not an edge appears between two nodes. The recent paper~\cite{KaurRoellin} deals with a variant 
where the first step is replaced by a deterministic procedure, leading to lattice type random graphs.

The analogy to ranks and permutations suggests the transfer of known results for the latter to
the graph situation considered here. Notably, classical efficiency concepts for the sample situation, such as those
related to the names of Bahadur, Pitman, Hodges and Lehmann, once transferred to the graphon situation, should
lead to quantitative comparisons of the above tests that go beyond simulations; see~\cite{BaGrEJS3} for such 
results for tests of independence based on pattern counts. Efficiency concepts are often based on large deviation
results. 
In the permutation
situation the theoretical efficiency results in \cite{BaGrEJS3} make use of large deviation results  for 
$U$-statistics. The theory of  large deviations for random graphs is a more recent development,  
see~\cite{Chatter} for a self-contained introduction and large deviation results
for Erd{\H os}-R\'enyi graphs.

Finally, substructure counts could possibly be 
used in the two-sample situation, meaning here a test of the hypothesis that two graphs are generated 
from the same graphon; see~\cite{BaGrBernoulli2025} for the permutation situation. This could be interpreted as
a statistical version of the graph isomorphism problem. The transfer would now require the extension
of results on distributional limits from the situation considered here to general graphons, again an area of 
current research.

\bibliographystyle{alpha } 

\begin{thebibliography}{}

\bibitem{BaGrBernoulli2025} {\sc Baringhaus, L., Gr\"ubel, R.} (2025)
  Pattern-based tests for two-dimensional copulas.
  {\em Bernoulli} 31, 3034--3059.

\bibitem{BaGrEJS3} {\sc Baringhaus, L., Grübel, R.} (2026)
Efficiency of pattern-based independence tests.
{\em Electron. J. Stat.} 20 (1), 1907-1942.
	
\bibitem{BergsmaDassios} {\sc Bergsma, W., Dassios, A.} (2014)
A consistent test of independence based on a sign covariance related to {Kendall}'s tau. 
{\em Bernoulli} 20, 1006--1028.

\bibitem{BFJ}
{\sc Brune, B., Flossdorf, J.  and Jentsch, C.} (2023)
 Goodness-of-fit testing based on graph functionals for homogeneous {Erd{\H{o}}s}-{R{\'e}nyi} graphs.
 {\em Scand. J. Stat.} 52, 332--380.

\bibitem{Chan}{\sc Chan, T. F. N., Kr{\'a}l, D., Noel, J. A.,  Pehova, Y.,	Sharifzadeh, M., Volec, J.} (2020)
Characterization of quasirandom permutations by a pattern sum.
{\em Random Struct. Algorithms} 57, 920--939.

\bibitem{Chatter}{\sc Chatterjee, S.} (2017) {\em Large deviations for random graphs}.
 Lecture Notes in Mathematics 2197. École d’Été de Probabilités de Saint-Flour. 
Springer, Cham. 

\bibitem{CGW}{\sc Chung, F.R.K., Graham, R.L., Wilson, R.M.} (1989)
Quasi-random graphs.  {\em Combinatorica} 9, 345--361.

\bibitem{EGW2}
{\sc Evans, S.N., Gr{\"u}bel, R., Wakolbinger, A.} (2017)
Doob--{M}artin boundary of {R}{\'e}my's tree growth chain.
{\em Ann. Probab.} 45,  225--277.

\bibitem{Hoeffding} {\sc Hoeffding, W.} (1948)
A non-parametric test of independence.
{\em Ann. Math. Stat.} 19, 546--557.

\bibitem{JansonGHS} {\sc Janson, S.} (1997)
{\em Gaussian Hilbert spaces}. Cambridge University Press, Cambridge.

\bibitem{JansonQuasi} {\sc Janson, S.}(2011) 
Quasi-random graphs and graph limits.
{\em Europ. J. of Combinatorics} 32, 1054–-1083

\bibitem{JLR} {\sc Janson, S., {\L}uczak, T. and Ruci{\'n}ski, A.} (2000)
 {\em Random graphs}.
Wiley, New York. 

\bibitem{KaurRoellin}{\sc Kaur, G. and Röllin, A.}(2021)
Higher-order fluctuations in dense random graph models. 
{\em Electron. J. Probab.} 26, 1--36. 

\bibitem{Lovasz} {\sc Lov{\'a}sz, L.},
 {\em Large networks and graph limits} (2012)
 Colloq. Publ., Am. Math. Soc. 60.

\bibitem{Oua} {\sc Ouadah, S., Robin, S. and Latouche, P.} (2020)
 Degree-based goodness-of-fit tests for heterogeneous random graph models: independent and 
exchangeable cases.
 {\em Scand. J. Stat.} 47, 156--181.


\bibitem{Yana}
{\sc Yanagimoto, T.} (1970) On measures of association and a related problem.
{\em Ann. Inst. Stat. Math.} 22, 57--63.

\bibitem{Zhao} {\sc Zhao, Yufei} (2023) 
{\em Graph Theory and Additive Combinatorics—Exploring Structure and Randomness.}
Cambridge University Press, Cambridge.

\end{thebibliography}
{}

\end{document}